\documentclass[reqno]{amsart}
\usepackage[T1]{fontenc}
\usepackage[utf8]{inputenc}
\usepackage[hang]{caption}

\usepackage{enumerate,amsmath,amssymb}
\usepackage{xcolor, hyperref,fullpage, tikz}
\usepackage{footnote}

\theoremstyle{plain}      
 
\newtheorem{thm}{Theorem}[section]     
\newtheorem{theorem}[thm]{Theorem}    
     
\newtheorem{corollary}[thm]{Corollary}     
     
\newtheorem{lemma}[thm]{Lemma}

\theoremstyle{definition}

\newtheorem{definition}[thm]{Definition}
\newtheorem{notation}[thm]{Notation}   

\newtheorem{remark}[thm]{Remark}
\newtheorem{example}[thm]{Example}

\renewcommand{\epsilon}{\varepsilon}
\let\theta\vartheta
\let\phi\varphi

\let\<\langle
\let\>\rangle

\let\doba\mathbb
\def\mC{\doba{C}}
\def\mH{\doba{H}}

\def\mR{\doba{R}}
\def\mS{\doba{S}}

\def\scal{{\mathop{\rm scal}}}
\def\Ric{{\mathop{\rm Ric}}}

\def\Spin{{\mathop{\rm Spin}}}
\def\SO{{\mathop{\rm SO}}}
\def\Si{\Sigma}

\def\vol{{\mathop{\rm vol}}}
\def\Id{\operatorname{Id}}

\let\pa\partial

\let\la\lambda

\gdef\Id{\mathrm{Id}}

\let\C\mC
\let\R\mR

\newcommand{\Q}{{Q^{*}}}

\newcommand{\vo}{\mathrm{dvol}}
\newcommand{\ds}{\mathrm{ds}}
\renewcommand{\d}{\mathrm{d}}

\renewcommand{\Re}{\mathrm{Re}}

\usepackage{stmaryrd}   
\newcommand{\kuno}{\varowedge}
\begin{document}
\makeatletter
\newtheorem*{rep@theorem}{\rep@title}
\newcommand{\newreptheorem}[2]{%
\newenvironment{rep#1}[1]{%
 \def\rep@title{#2 \ref{##1}}%
 \begin{rep@theorem}}%
 {\end{rep@theorem}}}
 \makeatother
\newreptheorem{corollary}{Corollary}

\title{Positive mass theorem for some asymptotically hyperbolic manifolds}

\author{Bernd Ammann} 
\address{Fakult\"at f\"ur Mathematik, Universit\"at Regensburg, 93040
Regensburg, Germany}
\email{bernd.ammann@mathematik.uni-regensburg.de}

\author{Nadine Gro\ss e} 
\address{Institut f\"ur Mathematik, Universit\"at Leipzig, 04109 Leipzig,
Germany}
\email{grosse@math.uni-leipzig.de}




\begin{abstract}
We  prove a positive mass theorem for some noncompact spin manifolds that are asymptotic to products of hyperbolic space with a compact manifold. As conclusion we show the Yamabe inequality for some noncompact manifolds which are important to understand the behaviour of Yamabe invariants under surgeries.
\end{abstract}

\maketitle
  
\section{Introduction}

Let $(M, g)$ be a Riemannian manifold of dimension $m\geq 3$.  For compactly supported smooth functions $u\colon M\to \mR$ we define \[ Q^*_M (u):=\frac{\int_M uL_gu\, \vo_g}{\Vert u\Vert_p^2}\ \text{and}\ \Q(M,g):=\inf_u Q^*_M (u)\]
where $p=\frac{2m}{m-2}$ and $L_g=a_m\Delta_g+\scal_g$ is the conformal Laplacian, $a_m=4\frac{m-1}{m-2}$ . The quantity $\Q(M,g)$ is a conformal invariant -- the well-known Yamabe constant that was defined in order to tackle the Yamabe problem: Does every compact Riemannian manifold admit a conformal metric with constant scalar curvature?

One of the main step in the solution of the Yamabe problem is to prove the inequality $\Q(M^m)<\Q(\mS^m)$ for compact manifolds not  conformally diffeomorphic to the standard sphere. For $m\geq 6$ and $M$ not conformally flat, this was proven with a test function supported in an arbitrarily small ball around a point with non-vanishing Weyl curvature, see \cite[p.~292]{Aubin}. This argument also holds for  noncompact manifolds that are not conformally flat in the case $m\geq 6$. 

The remaining cases for compact $M$ were solved by Schoen in  \cite{schoen_84} using the positive mass theorem \cite{schoen_79, schoen_yau_88}.
It is natural to ask whether this also holds for noncompact manifolds. The special case of $\mS^n\times \mH_c^{k+1}$ was needed in \cite{ammann.grosse:p13b}. 
In the present article we prove 

 \begin{corollary}\label{repcor}Let $m=n+k+1$, $m\geq 3$, $k>0$, and $c\in [0,1)$. Then
\[\Q(\mS^n\times \mH_c^{k+1},\sigma^n+g_c)<\Q(\mS^m, \sigma^m).\]
\end{corollary}

To prove this corollary we have to establish for such manifolds with $m=3,4,5$ a positive mass theorem. 

\begin{theorem}\label{mainthm}
Let  $(Z^m=N^n\times \mH_c^{k+1}, g=g_{N}+ g_c)$ where $(\mH_c^{k+1}=\mS^k\times [0,\infty), g_c=\sinh_c(r)^2\sigma^k+dr^2)$ with $\sinh_c(r):= \frac{1}{c}\sinh(cr)$ for $c>0$, $\sinh_0(r):=r$ and $\sigma^k$ denotes the standard metric on $\mS^k$. Assume that $N^n$ is a closed Riemannian spin manifold with constant scalar curvature and assume that $m=3,4,5$. 
Assume that 
\begin{equation}\label{cond_main_1}  \scal_N> c^2 k\frac{n-1}{m-2}.\end{equation} Then the mass of $(Z,g)$ is nonnegative.

Moreover, if the mass of $(Z,g)$ is zero, then $(Z,g)=\mS_c^{m-1}\times \mR$ or $(Z,g)=\mS_c^{m-k-1}\times \mH_c^{k+1}$, $k> 0$ and $c\in (0,1]$, where $\mS_c^{n}$ is the rescaled standard sphere  with sectional curvature $c^2$.
\end{theorem}

In Theorem~\ref{mainthm_2} we prove a more general version of the theorem above where we allow that are asymptotic to manifolds $Z$ as above assuming a certain symmetry and for arbitrary dimension $m\geq 3$. The basic idea of the proof follows Witten's arguments for proving the positive mass theorem for compact spin manifolds. In the proof we need some decay estimates for the Green functions which will be given in Section~\ref{Green_est_sec}.\\

{\bf Acknowledgment.} Most of the work on this article was done during the second author's stay at the University of Regensburg supported by the Graduiertenkolleg 'Curvature, Cycles, and Cohomology'. The second author thanks the institute for its hospitality.

\section{Preliminaries}

\subsection{Notations}

In the article  a spin manifold always means a manifold admitting a spin structure together with a fixed choice of spin structure. The notion of spin structures can be defined for arbitrary oriented manifolds, but as soon as we have a Riemannian metric it yields a spin structure in the sense of $\Spin (n)$-principal bundles.

For a Riemannian spin manifold $(M,g)$ we will always  write $\Sigma_M$ for the spinor bundle. The Dirac operator on $(M,g)$ is denoted by $D^g$.

Moreover, $\bar{B}_r(x)$ denotes the closed ball around $x\in M$ of radius $r$ w.r.t. the metric $g$.

Furthermore, $\mS^k$ always denotes the  sphere with sectional curvature $1$ and the corresponding metric is denoted by $\sigma^k$.

\subsection{Green functions and ADM mass}\label{subsec_Green}

In this section, we collect existence results for the Green function of the conformal Laplacian and the Dirac operator and give the definition of the mass which essentially goes back to Arnowitt, Deser and Misner.

\begin{theorem}
 Let  $(M^m, g)$ be a complete Riemannian spin manifold with positive injectivity radius. Assume that for an $r>0$ that is smaller than the injectivity radius there are constants $C_1,C_2>0$ with $C_1\leq \vol(B_r(x), g)\leq C_2$ for all $x\in M$. 
  Let the  Dirac operator $D^g$ be invertible. Then, $D^g$ posesses a unique Green function, i.e., there is a smooth section $G\colon M\times M\setminus  \Delta\to \Si_M\boxtimes \Si_M^*$ that is locally integrable and for any $x\in M$, $\psi_0\in \Si_M|_y$, and $\phi\in C_c^\infty(\Si_M)$
\[ \int_M \<G(y,x)\psi_0, D^g\phi(y)\>\, \d y=\<\psi_0, \phi(x)\>.\]
 Here $\Delta:=\{ (x,x)\in M\times M\}$ is the diagonal in $M\times M$.
\end{theorem}

\begin{proof} For a reference see \cite[Prop.~3.6]{ammann.grosse:p13a}. In the proof therein of the existence of a Green function  we use the assumption of bounded geometry only for uniform volume bounds and the positive injectivity radius which we directly put in here as an assumption.
\end{proof}

\begin{theorem}\label{exist_Green_L} Let $(M^m, g)$ be a complete Riemannian manifold with  positive injectivity radius. Assume that for an $r>0$ that is smaller than the injectivity radius there are constants $C_1,C_2>0$ with $C_1\leq \vol(B_r(x), g)\leq C_2$ for all $x\in M$.  Let the conformal Laplacian $L_g$ be invertible. Then $L_g$ admits a unique Green function $\Gamma\colon M\times M\setminus \Delta \to \R$ such that $\Gamma(x,.)\in L^2(M\setminus B_r(x))\cap C^\infty(M\setminus  B_r(x))$ for all $x\in M$ and all $r>0$, i.e., for all $v\in C_c^\infty(M)$
  \[ \int_M \Gamma(x,y)L_gv(y)\, \d y= v(x).\]
Moreover, $\Gamma$ is everywhere positive.

If we assume additionally that on an open subset $U$ the metric $g$ is conformally flat and $p_0\in U$, then the Green function has the following expansion in conformal normal coordinates as $x\to p_0$
\begin{equation}\label{green_exp} \Gamma(x, p_0)=\frac{1}{(m-2)\omega_{m-1} r^{m-2}}+ m_{p_0} +o(1),\end{equation}
where $r={\rm dist}_g(x, p_0)$, $\omega_{m-1}$ is the volume of $\mS^{m-1}$, and $m_{p_0}\in \mR$.
\end{theorem}

\begin{proof} The proof for existence of the Green function is done analogously as for the Dirac operator, cf. \cite[Prop.~3.6]{ammann.grosse:p13a}.  For the expansion of the Green function in normal coordinates see \cite[Prop.~2.1]{HH_14}.
\end{proof}

\begin{definition}\label{mass_d}
 The constant $m_{p_0}$ in \eqref{green_exp} is called the mass at $p_0$.

\begin{center}
\framebox{\parbox{0.8\linewidth}{Fix $p_0\in M$ and $\psi_0\in \Si_M|_{p_0}$. For the rest of the article we set
\[u(.):=\Gamma(.,p_0)\ \text{and\ } \psi(.):=G(., p_0)\psi_0.\]}}
\end{center}

\end{definition}

\subsection{Symmetries}\label{subsec_symm}

In this section we introduce the manifolds $M_Z$ we merely want to consider in this paper. Moreover, we collect the implications of the symmetry assumptions on $M_Z$ on the Green functions.

\begin{definition}[Model spaces at infinities] Let $(N^n, g_N)$ be a closed Riemannian spin manifold of dimension $n\geq 0$, and let $f\colon [a,\infty)\to \R$ be a smooth positive function such that $f(r)\to \infty$ and  $ f^{-1}(r)f''(r)\to c^2$  as $r\to \infty$ for some $c\geq 0$. Set \[(Z^m, g_Z)= (N\times \mS^k \times (a,\infty), g_{N}+f(r)^2\sigma+dr^2)\]
with $m=n+k+1$.
\end{definition}

\begin{remark}
Analogously as in \cite[Sect. 4.2]{ammann.grosse:p13a} the action of $\SO(k+1)$ on $\mS^k$ induces an isometric action on $Z$ such that it lifts to a $\Spin(k+1)$-action on the spinor bundle.
\end{remark}

\begin{example}
 For $f(r)=\sinh(r)$ and $a=0$ we obtain $Z=\mH^{k+1}\times N$ with $c=1$. For $f(r)=r$ and $a=0$ we get $c=0$ and $Z=\R^{k+1}\times N$.
\end{example}

\begin{remark} The conditions $f(r)\to \infty$ and $ f^{-1}(r)f''(r)\to c^2$ imply  $f'(r)f^{-1}(r)\to c$ as $r\to \infty$. In order to see this, we rewrite this as a first order ODE system and use Theorem~\ref{pert_ode}. Thus, by $\scal_g=\scal_{N} +f(r)^{-2} \scal_{{\mS^k}} -n(n-1) f(r)^{-2} (\partial_r f(r))^2 -2nf(r)^{-1}\partial_r^2 f$, cf.  \cite[Thm. 2.1]{Dobarro_LamiDozo_87},  $\scal_g\to s:=\scal_{N}-k(k+1)c^2$ as $r\to \infty$. Moreover, the unnormalized mean curvatures of $N\times \mS^k$ and $\mS^k$ in $Z$ are both given by  $H(r)=f'(r)f^{-1}(r)$. 
\end{remark}

\begin{notation}\label{def_M} From now on assume that $(M_Z, g)$ is an $m$-dimensional connected Riemannian spin manifold with an isometric $\SO(k+1)$-action with a lift to a $\Spin(k+1)$-action on the spinor bundle. Moreover, $M_Z$  shall be $\Spin(k+1)$-equivariantly spin isometric to $(Z, g_Z)$  outside a compact subset.
\end{notation}

\begin{remark}
Note that by $f^{-1}f''\to c^2$ and $f\to \infty$, the manifold has positive injectivity radius and its curvature tensor is uniformly bounded. Thus, by comparison theorems we have the uniform volume bounds required in  the results on the Green function in Subsection~\ref{subsec_symm}.
\end{remark}

\begin{lemma}\label{green_symm} Let the conformal Laplacian on $M_Z$ be positive, and let the Dirac operator on $M_Z$ be invertible.
Then, the Green function $\Gamma$ of the conformal Laplacian on $M_Z$ is $\SO(k+1)$-equivariant and, thus,  $u|_Z$, cf. Definition~\ref{mass_d}, is only a function in $r$. Moreover, the Green function of the Dirac operator on $M_Z^m$ is $\Spin(k+1)$-equivariant. Then,  we have $(D^{\mS^k})^2\psi|_Z=\frac{k^2}{4}\psi|_Z$ for $\psi$ as in Definition~\ref{mass_d}  and where the $D^{\mS^k}$ is the Dirac operator along the $\mS^k$ component in $Z$. 
\end{lemma}

\begin{proof}
Since the Green function and the conformal Laplacian are both $\SO(k+1)$-invariant, it follows by the uniqueness of the Green function that $u|_Z$ is only a function of $r$. 

The claim for $\psi$ is proven completely analogously to \cite[Lem.~5.3]{ammann.grosse:p13a}.
\end{proof}

\subsection{Perturbation of linear systems of ODE}

In order to estimate the Green functions in Section~\ref{Green_est_sec} we need the asymptotic behaviour for certain perturbed differential equations.
We cite here a result that can be found in more general versions, e.g. in \cite[I.2 Thm. (*)]{hartman_winter}, \cite[Satz~13]{perron}.

\begin{theorem}\label{pert_ode}  
Let $x(t)$ denote a function in $t\in [t_0,\infty)$ with values in $\C^d$. Let $J\in M_d(\mC)$, let $G\colon [t_0,\infty)\to M_d(\mC)$ be continuous where $M_d(\mC)$ denotes the sets of $d\times d$ complex matrices. Moreover let $G(t)\to 0$ as $t\to \infty$. 
Let $r_1, \ldots,  r_d$ be the real parts of the eigenvalues of $J$ -- listed with multiplicities.  Then the differential equation
\begin{align*}
 x'=(J+G(t))x
\end{align*}
has linearly independent solutions $x_1(t),\ldots , x_d(t)$ with the following property: For all $\epsilon>0$ there is $t_1$ such that for all $t\geq t_1$
\begin{align*} |x_i(t_1)|e^{ (r_i-\epsilon) (t-t_1)} \leq |x_i(t)|\leq |x_i(t_1)|e^{ (r_i+\epsilon) (t-t_1)}.
\end{align*}
\end{theorem}

\section{Positive mass theorem and decay of the Green functions}

In \cite{witten} E.~Witten  gave a proof of the positive mass conjecture for spin manifolds. A simplified version of this proof was given by E.~Humbert and the first author in \cite{ammann.humbert:05}. In this section we examine conditions such that the result remains true for general noncompact spin manifolds. For that we need to assume decay conditions for the involved Green functions in order to show convergence of some integrals at infinity.

\begin{theorem}\label{pos_gen} Let $(M, g)$, $p_0\in M$, $u$ and $\psi$ as in Definition~\ref{mass_d}. Assume that if ${\rm dim}\geq 6$ the metric is conformally flat in a neighbourhood of $p_0$. Let $K$ be a compact subset of~$M$, and $R\colon M\setminus K \to (0,\infty)$ be a smooth proper function such that for all $r>0$ the set $S_r:=\{x\in M\setminus K\ |\ R(x)=r\}$ is a smooth hypersurface in $M$ and $B_r:=K\cup \bigcup_{r'\leq r} S_{r'}$ is a compact subset of $M$ with boundary $S_r$.  Let $\nu$ be the outer normal of $S_r=\partial B_r$. Assume that there are constants $C_i>0$ ($i=1,2,3$), $r_0>0$, $\alpha_1\geq\alpha_2>0$, $\beta>0$ such that for all $x\in S_r $ with  $r>r_0$ we have
\[|u(x)|\geq C_1 e^{-\alpha_1 r}\quad |\partial_\nu u(x)|\leq C_2 e^{-\alpha_2 r}\]
and \[|\psi(x)|\leq C_3 e^{-\beta r}\quad |\nabla_\nu \psi(x)|\leq C_3 e^{-\beta r}\]
and \begin{equation}\label{vgl}\frac{m}{m-2}\alpha_1 -\alpha_2 < 2\beta.\end{equation}
Then, the mass at $p_0$ is nonnegative. If the mass is zero, then $M\setminus \{p_0\}$ carries a conformal metric that is flat and admits a basis of parallel spinors.
\end{theorem}

\begin{proof} We use the notations given in Subsection~\ref{subsec_Green} and follow the idea of the proof of \cite[Thm. 2.2]{ammann.humbert:05}.
Note that in the published version of \cite{ammann.humbert:05} a small term was omitted, see the arXiv version.

 Set $\tilde{g}=u^{\frac{4}{m-2}}g$ and $\tilde{\psi}=u^{-\frac{m-1}{m-2}}\psi$. Then $(M\setminus \{p_0\}, \tilde{g})$ has zero scalar curvature and $\tilde{D}\tilde{\psi}=0$ on $M\setminus\{p_0\}$. Thus, by the Schr\"odinger-Lichnerowicz formula $\tilde{\nabla}^*\tilde{\nabla}\tilde{\psi}=0$. Moreover, $\partial_{\tilde{\nu}}=-u^{-\frac{2}{m-2}}\partial_\nu$ has length one w.r.t. $\tilde{g}$.

As in \cite[Proof of Thm. 2.2]{ammann.humbert:05}, using $\tilde{\nabla}^*\tilde{\nabla}\tilde{\psi}=0$ and partial integration we obtain
\[ 2\int_{M\setminus B_p(\epsilon)} |\tilde{\nabla}\tilde{\psi}|^2 \vo_{\tilde{g}}= \int_{S_{p_0}(\epsilon)} \partial_{\tilde{\nu}} |\tilde{\psi}|^2 \ds_{\tilde{g}}+ \lim_{r\to \infty}\int_{S_{p_0}(r)}\partial_{\tilde{\nu}}|\tilde{\psi}|^2\ds_{\tilde{g}}.\]

If  $\lim_{r\to \infty}\int_{S_{p_0}(r)}\partial_{\tilde{\nu}}|\tilde{\psi}|^2\ds_{\tilde{g}}=0$, then the nonnegative mass follows as in the compact case, see \cite[Proof of Thm. 2.2 and Sect. 3]{ammann.humbert:05}. Here we restrict to prove the vanishing of the limit from above: We have
\begin{align*}
 \int_{S_{p_0}(r)} \partial_{\tilde{\nu}}|\tilde{\psi}|^2\ds_{\tilde{g}}=&  -\int_{S_{p_0}(r)} u^2\partial_{{\nu}}|u^{-\frac{m-1}{m-2}}{\psi}|^2\ds_{g}\\
 =&-\int_{S_{p_0}(r)} u^{-\frac{2}{m-2}}\partial_{{\nu}}|\psi|^2\ds_{g}+2\frac{m-1}{m-2}\int_{S_{p_0}(r)} u^{-\frac{m}{m-2}}|{\psi}|^2\partial_{{\nu}}u \ds_{g}.
\end{align*}

Using the decay estimates we obtain

\begin{align*}
\left| \int_{S_{p_0}(r)} u^{-\frac{2}{m-2}}\partial_{{\nu}}|\psi|^2\ds_{g}\right|\leq C e^{(\frac{2}{m-2}\alpha_1 -2\beta)r}
\end{align*}

and
\begin{align*}
\left|\int_{S_{p_0}(r)} u^{-\frac{m}{m-2}}|{\psi}|^2\partial_{{\nu}}u \ds_{g}\right|\leq C e^{(\frac{m}{m-2}\alpha_1 -\alpha_2 -2\beta)r}.
\end{align*}

Since $\frac{m}{m-2}\alpha_1 -\alpha_2 <2\beta$ implies that $\frac{2}{m-2}\alpha_1 <2\beta$, both integrals go to zero as $r\to \infty$. Thus, the mass at $p_0$ is nonnegative.

Assume now that the mass is zero, then as in the compact case $\tilde{\nabla}\tilde{\psi}=0$, cp. \cite[Thm. 2.2]{ammann.humbert:05}. Moreover, since the choice of $\psi_0$ in the definition of $\psi$ is arbitrary, we again obtain a basis of parallel spinors on $(M\setminus \{p_0\}, \tilde{g})$. Thus, $(M\setminus \{p_0\}, \tilde{g})$ is flat, and $(M\setminus \{p_0\}, g)$ is conformally flat.
\end{proof}

\section{\texorpdfstring{Green function estimates for $M_Z$.}{Green function estimates for MZ.}}\label{Green_est_sec}

We consider a spin manifold $(M_Z^m,g)$ as introduced in Notation~\ref{def_M}. The aim of this section is to provide estimates for the corresponding Green functions as required to apply Theorem~\ref{pos_gen}. 

 \subsection{\texorpdfstring{Estimating the Green function of $L_g$}{Estimating the Green function of Lg}}

\begin{lemma}\label{GreenLapl} Let $(M_Z^m, g)$ be as in Notation~\ref{def_M}. Let the conformal Laplacian $L_g$ of $(M_Z,g)$ be invertible. We use the definitions and notations from Subsection~\ref{subsec_Green}.

Then for every $\epsilon >0$  there is an $r_0>0$  such that for all $x=(x_N, x_{\mS}, r)\in N\times \mS^k\times [r_0,\infty)$ we get
\[ u(x)\geq C_0 e^{-\alpha_+ r}  \text{\ and\ } |\partial_r u(x)|\leq C_1 e^{-\alpha_- r},\]
where \[\alpha_\pm= \frac{kc}{2}\pm\epsilon +\Re\sqrt{\frac{k^2c^2}{4}+\frac{\scal_\pm-c^2k(k+1)}{a_m}},\]
$\scal_+=\sup_N \scal_N$, $\scal_-=\inf_N \scal_N$,  and $C_0,C_1>0$. 
\end{lemma} 
 
  \begin{proof}  On $M\setminus \{p_0\}$ we have $L_gu=0$ which reads 
 \begin{equation}
 \label{u_gen} 
 \frac{\partial^2}{\partial^2 r} u -\Delta^{N} u -\frac{1}{f(r)^2} \Delta^{{\mS^k}} u + k\frac{f'(r)}{f(r)} \frac{\partial}{\partial r} u -\frac{\scal_g}{a_m} u=0
 \end{equation}
 
  where $\scal_g=\scal_{N} +f(r)^{-2} \scal_{{\mS^k}} -n(n-1) f(r)^{-2} (\partial_r f(r))^2 -2nf(r)^{-1}\partial_r^2 f$, cf.  \cite[Thm. 2.1]{Dobarro_LamiDozo_87}. Note that by the assumptions on $f$ and its derivatives we have $\scal_g\to \scal_{N}-k(k+1)c^2:=s$ as $r\to \infty$.
  
   We decompose the space of smooth functions on $N\times \mS^k$ into minimal subspaces which are generated by common eigenfunctions of the commuting operators $\Delta^N$ and $\Delta^{\mS^k}$. If we extend those eigenfunctions constant in $r$-direction, we obtain a decomposition of the space of smooth functions on $N \times \mS^k\times [a, \infty)$.
   In that sense we decompose $u=\sum_{i,l} u_{i,l}$ where $\Delta^{N}u_{i,l}=\mu_iu_{i,l}$ and $\Delta^{{\mS^k}}u_{i,l}=\lambda_lu_{i,l}$ where $0=\mu_0< \mu_1\leq \ldots$ and $0=\lambda_0< \lambda_1\leq \ldots$. 
 Then, \eqref{u_gen} decomposes into the equations
\begin{equation}\label{equ_uil} \frac{\partial^2}{\partial^2 r} u_{i,l} + k\frac{f'(r)}{f(r)}\frac{\partial}{\partial r} u_{i,l}-\left(\mu_i + \lambda_l f(r)^{-2}+\frac{\scal_g}{a_m}\right) u_{i,l}= 0.
  \end{equation}
  
which can be written as 

\[\frac{\partial}{\partial_r}
 \begin{pmatrix}
  u_{i,l}\\ \frac{\partial}{\partial_r}u_{i,l}
 \end{pmatrix}
 = \Bigg( \underbrace{\begin{pmatrix}
           0&1\\
           \mu_i+\frac{s}{a_m}& -kc
          \end{pmatrix}}_{=:J}
+
 \underbrace{\begin{pmatrix}
           0&0\\
           \frac{s-\scal_g}{a_m}+f(r)^{-2} \lambda_l & k\left(c-\frac{f'(r)}{f(r)}\right)
          \end{pmatrix}}_{=:R(r)}
 \Bigg)  \begin{pmatrix}
  u_{i,l}\\ \frac{\partial}{\partial_r}u_{i,l}
 \end{pmatrix}
\]

Since $u\in L^2$, each $u_{i,l}\in L^2$ as well. Multiplication of \eqref{equ_uil} with $u_{i,l}$ and partial integration shows that $\partial_r u_{i,l}\in L^2$. 

Moreover, note that $J$ is constant and that $R(r)\to 0$ as $r\to \infty$. The matrix $J$ has the eigenvalues $\alpha_{i,\pm}= -\frac{kc}{2} \pm \sqrt{\frac{k^2c^2}{4}+ \mu_i+\frac{s}{a_m}}$. Fix $\epsilon>0$. Thus, by Theorem~\ref{pert_ode} and since $u_{i,l}$ and $\partial_r u_{i,l}$ are $L^2$ we obtain an $r_0>0$ such that for all $r\geq r_0$

\[ |u_{i,l}(r_0)| e^{(\Re\  \alpha_{i,-}- \epsilon)(r-r_0)}\leq  |u_{i,l}(r)|\leq |u_{i,l}(r_0)| e^{(\Re\  \alpha_{i,-}+\epsilon)(r-r_0)}
 \]
 and the analogous estimate for $\partial_r u_{i,l}$.

 By Lemma~\ref{green_symm} only $l=0$  appears. Thus, only the different eigenvalues in $N$-direction need to be taken into account. Set $u_{i}:=u_{i,0}$.  
  Note that $\int_N u_{i}=0$ for all $i>0$. But since $u$ is everywhere positive this implies that $u_0$ cannot be identically zero. Setting  $\alpha_+:=\sup_N (-\Re\,  \alpha_{0,-})$, this concludes the proof of the first estimate.

For the second estimate we first use that $\partial_ru\in L^2$ and set $\alpha_{i}^-=\sup_N \Re\, \alpha_{i,-}$ to obtain 
\begin{align*}
 \infty> \int_{Z} |\partial_r u|^2\geq &C \sum_{i} \int_{r_0}^\infty |\partial_r u_{i}|^2 (f(r))^{k}\d r\\
  \geq & C' \sum_{i} \int_{r_0}^\infty |\partial_r u_{i}(r_0)|^2 e^{\left(ck+2\alpha_{i}^- -2\epsilon\right) (r-r_0)}\d r\\
  \geq & C'' \sum_{i} |\partial_r u_{i}(r_0) |^2 \frac{1}{-2\alpha_{i}^- -ck+2\epsilon}>0
 \end{align*}
 for positive constants $C$, $C'$, $C''$.
 
 Now we can estimate $\partial_r u$ by using the triangle inequality, Cauchy-Schwarz inequality und the estimates from above:
 
 \begin{align*}
  |\partial_r u|\leq &\sum_{i} |\partial_r u_{i}(r_0)|e^{(\alpha_i^-+\epsilon)(r-r_0)}\\
  \leq & \Bigg(\underbrace{\sum_{i} |\partial_r u_{i}(r_0)|^2 \frac{1}{-2\alpha_i^--ck+2\epsilon}}_{<\infty}\Bigg)^\frac{1}{2}\!\!  \left(\sum_{i} (-2\alpha_i^--ck+2\epsilon)e^{(2\alpha_i^-+2\epsilon)(r-r_0)}\right)^\frac{1}{2}\\
  \leq & C e^{2(\alpha_{0}^-+\epsilon)r}\Bigg(\underbrace{\sum_{i} (-2\alpha_i^-+2\epsilon)-e^{2(\alpha_i^--\alpha_{0}^-)(r-r_0)}  
   }_{=:A}\Bigg)^\frac{1}{2}.
\end{align*}

  If we can show that $A$ is bounded, then the upper bound is established for $-\alpha_-:=\alpha_0^-$: 
Since $\alpha_{i}^->\alpha_{0}^-$ for all $i>0$, it is enough to bound $\sum_{i} (-2\alpha_i^- +2\epsilon)e^{2\alpha_i^-(r-r_0)}$.

Let $N(x):=\left|\{ l\ |\ \lambda_l\leq x \}\right|$. Then,  $N(\mu_i-1)\leq i\leq N(\mu_i)$. Together with Weyl's asymptotic law
there is an $i_0>0$ and an $\delta>0$ such that for all  $i\geq i_0$, $x\geq \mu_{i_0}$
  we have 
  \[2\sqrt{\mu_i}\geq -\alpha_i^-\geq \sqrt{\mu_i}\ \text{and}\ |N(x) x^{-k/2}-c|\leq \delta\](where $c$ depends on the geometry of ${\mS^k}$).
  Thus,  $(c-\delta)\mu_i^{k/2}\leq N(\mu_i)\leq i+1\leq 2i$ and $(c+\delta)\mu_i^{k/2}\geq N(\mu_i)\geq i$ and we obtain for $r>r_1>r_0$
  \[ \sum_{i\geq i_0} (-2\alpha_i^- +2\epsilon)e^{2\alpha_i^-(r-r_0)}\leq 2 \sum_{i\geq i_0} (2\sqrt{\mu_i}+\epsilon) e^{-2\sqrt{\mu_i}(r_1-r_0)} \leq c_0 \sum_{i\geq i_0} (i^{1/k}+c_1) e^{-c_2 i^{1/k}}
  \]
   where $c_0, c_1, c_2$ are positive constants. The last sum converges since the integral $\int_R^\infty (x^{\frac{1}{k}}+c_1) e^{-c_2 x^{\frac{1}{k}}}\, dx$ is finite for $R>0$.
Thus, for $r\geq r_1$ the term $A$ is bounded.
 \end{proof}

 \subsection{Green function of the Dirac operator}
 
 Similarly as in the last subsection we want to estimate the Green function of the Dirac operator on $(M_Z, g)$. Again we use the notations and definitions of Subsection~\ref{subsec_Green}. 
 
As in \cite[Sect.~6]{ammann.grosse:p13a} we decompose the space of spinors restricted to $N\times {\mS^k} \times \{r_1\}$
into complex subspaces of minimal dimensions which are invariant under 
$D^{N}$, $D^{{\mS^k}}$, and $\pa_r\cdot$. Such spaces have a basis of the form $\psi$, 
$\pa_r\cdot\psi$, $P\psi$, and $\pa_r\cdot P\psi$, where
$\psi$ satisfies $D^{N}\psi=\lambda \psi$, $(D^{{\mS^k}})^2\psi =\rho^2\psi$, 
 and $P:=D^{{\mS^k}}/\rho$ with $\lambda,\rho \in \mR$. 
All these operations commute with parallel transport in $r$-direction,
so by applying parallel transport in $r$-direction
we obtain spinors
$\psi$, $\pa_r\cdot\psi$, $P\psi$, and 
$\pa_r\cdot P\psi$ on $Z= N\times {\mS^k} \times [a,\infty)$ with similar relations,
and the space of all spinors of the form
\begin{equation*}
  \phi= \phi_1(r)\psi + \phi_2(r)\pa_r\cdot \psi +  \phi_3(r) P\psi + \phi_4(r)
\pa_r\cdot P\psi
\end{equation*}
is preserved under the Dirac operator $D$ on $Z$.

Then the operators discussed above restricted to such a minimal subspace are 
represented by matrices and the equation $D\phi=0$ on $Z$ reads as, cp. \cite[Proof of Prop.~6.2]{ammann.grosse:p13a}, 
 
\begin{equation*}
   \Phi_{\lambda, \rho}'(r)=  \Bigg(\underbrace{A -\frac{kc}{2}\Id}_{=:J}  +\underbrace{ \frac{k}{2}\Big(c-\frac{f'(r)}{f(r)}\Big)\Id  +  \frac\rho{f(r)} B}_{:=G(r)}\Bigg)
   \Phi_{\lambda, \rho}(r).
\end{equation*}

where $\Id$ is the identity matrix and 
\begin{equation*}
   A:=     
   \begin{pmatrix}
       0  & \la & 0 & 0 \\
       \la & 0 & 0 & 0\\
       0 & 0 & 0 & -\la\\
       0 & 0 & -\la &  0 \\
   \end{pmatrix},
   \qquad 
  B:=   \begin{pmatrix}
       0  & 0 & 0 & 1 \\
       0  & 0 & 1 & 0 \\
       0  & 1 & 0 & 0 \\
       1  & 0 & 0 & 0 \\
   \end{pmatrix},
\end{equation*}
 
The eigenvalues of $J$ are $\beta_{\lambda,\pm}=-\frac{kc}{2}\pm \lambda$ -- both of them have multiplicity two. Thus, by Theorem~\ref{pert_ode} and since $\Phi_{\lambda,\rho}$ is $L^2$, for each $\epsilon>0$ there is $r_0>0$ such that 

\begin{align*} |\Phi_{\lambda,\rho}(r_0)| e^{(\Re \beta_{\lambda,-}-\epsilon)(r-r_0) }\leq&  |\Phi_{\lambda,\rho}(r)|
\leq |\Phi_{\lambda,\rho}(r_0)| e^{(\Re \beta_{\lambda,-}+\epsilon)(r-r_0)}.
 \end{align*}
 
 By Lemma~\ref{green_symm} only $\rho^2=\frac{k^2}{4}$ occurs, i.e., $\rho=\pm \frac{k}{2}$. Thus, we can proceed completely analogously as for the estimate of $\partial_r u$ in the proof of Lemma~\ref{GreenLapl} in order to obtain the upper decay of $\psi$ and $\partial_r \psi$. 
 Then, in total we obtain the following
 \begin{lemma}
\label{GreenD} Let $(M_Z^m, g)$ be as in Notation~\ref{def_M}. Let the Dirac operator $D^g$ of $(M_Z,g)$ be invertible, and let $\lambda_N^2$, $\lambda_N\geq 0$, be the lowest Dirac eigenvalue for the square of the Dirac operator $(D^N)^2$ on $(N, g_N)$.
Then for $\epsilon >0$ and $r$ large enough we get for $x=(x_N,x_{\mS},r)\in Z$ the estimates (using Definition~\ref{mass_d})
\[ |\psi(x)|\leq C_0 e^{-\beta r}  \text{\ and\ } |\nabla_r \psi(x)|\leq C_1 e^{-\beta r},\]
where $\beta= \frac{kc}{2}+\lambda_N$.   
 \end{lemma}

\section{Proof of Theorem~\ref{mainthm}}

With the preparations on the Green function decay on $(M_Z, g)$ in the last two subsections we are now ready to apply Theorem~\ref{pos_gen}. 
 
\begin{theorem}\label{mainthm_2}
Let  $(M_Z^m,g)$ be a Riemannian spin manifold as described above, for a precise Definition see Notation~\ref{def_M}. 
If $m\geq 6$, assume additionally that there is a point $p\in M$ such that $(M_Z,g)$ is conformally flat on a neighbourhood of $p$.
We assume that the conformal Laplacian and the Dirac operator on $(M_Z, g)$ are invertible. 
Furthermore, let 
\begin{equation}\label{cond_main} \frac{kc(3-m)}{m-2} +\frac{m}{m-2}\Re \sqrt{b+\frac{\sup_N \scal_{N}}{a_m}}-\Re \sqrt{b+\frac{\inf_N \scal_{N}}{a_m}}<2\lambda_{N} \end{equation} where $\lambda_{N}^2$, $\lambda_{N}\geq 0$, is the smallest eigenvalue of the square of the Dirac operator, and 
\[ b:=\frac{kc^2(1-n)}{4(m-1)}.\] Then the mass of $(M_Z,g)$ is nonnegative.

Moreover, if the mass of $(M_Z,g)$ is zero, then $(M_Z,g)$ is conformally equivalent to  $\mS^{m-k-1}\times \mH^{k+1}$.
\end{theorem}

\begin{proof} First we note that $n>1$, since otherwise $L_g$ cannot be invertible. By Lemmata~\ref{GreenLapl} and~\ref{GreenD} we can apply Theorem~\ref{pos_gen} 
for 
\begin{align*}\alpha_1=& \frac{kc}{2}+\epsilon +\Re\sqrt{\frac{k^2c^2}{4}+\frac{\sup_N \scal_N-c^2k(k+1)}{a_m}}\\                                                                                                                                       \alpha_2=& \frac{kc}{2}-\epsilon +\Re\sqrt{\frac{k^2c^2}{4}+\frac{\inf_N \scal_N-c^2k(k+1)}{a_m}}\\
\beta=&\frac{kc}{2}+\lambda_N.
                                                                                   \end{align*}
 Since $\epsilon$ can be chosen arbitrarily small, \eqref{cond_main} is exactly the condition assumed in \eqref{vgl}. Thus, we obtain that the mass at $p_0$ is nonnegative.
  
Assume now that the mass is zero. Then, by Theorem~\ref{pos_gen} $M\setminus\{p_0\}$ is conformally flat.  By Lemma~\ref{confflat_prod} one of the following cases occurs (where $G$ is always finite):
\begin{center}
\begin{framebox}{
\begin{tabular}{c|c|c|c}
& $N$ & $Z$ & \\
\hline
&&&\\
1. & $\mR^n/G$ & $(\mR^n/G\times \mS^{k}\times (a,\infty), g_E+ c^2r^2\sigma^k +\d r^2) $ & $c\in [0,1]$, $n>1$ \\
2.& $\mS_c^n/G$ & $(\mS_c^n/G\times  \mS^{k}\times (a,\infty), \sigma_c^n+ \sinh_c(r)^2\sigma^k +\d r^2)$ & $n>1$, $c>0$\\
3.& $\mS_c^{m-1}/G$ & $\mS_c^{m-1}/G\times (a,\infty)$ &  $c>0$\\
4.& $\mH_c^{m-1}/G$ & $\mH_c^{m-1}/G\times (a,\infty)$ & $c>0$\\
\end{tabular}}\end{framebox}
\end{center}
The first and fourth case are excluded since in that case the conformal Laplacian on $M_Z$ is not invertible.

We now consider the second case: $Z=\mS_c^n/G\times (\mH_c^{k+1}\setminus \bar{B}_a(0))$ for $G$ finite, $n>1$ and $c>0$. Let $(\hat{M}= \mS_c^n\times (\mH_c^{k+1}\setminus \bar{B}_a(0)), \hat{g})$ be the $G$-cover of $M$. The corresponding projection is denoted by $\pi\colon \hat{M}\to M$. Then, $\hat{M}$ still fulfills the assumptions in Theorem~\ref{exist_Green_L} and, thus,  $L_{\hat{g}}$ possesses a Green function $\hat{\Gamma}$.  Since $\lambda_N=\lambda_{\mS^n}$, \cite[p. 62]{baer_diss}, $\hat{M}$ also fullfills the assumptions of this theorem. Thus, from above we know that the mass of $\hat{M}$ at $\hat{p}_0$ with $\pi(\hat{p}_0)=p_0$ is nonnegative. On the other hand together with \cite[Prop. 4.3]{schoen_yau_88} we obtain for nontrivial $G$ that $0=m_{p_0}>m_{\hat{p}_0}\geq 0$ which gives a contradiction. Hence, $G$ is the trivial group. Analogously, we obtain in the third case that $G$ is trivial. Summarizing  we know that $Z=\mS_c^n\times(\mH_c^{k+1}\setminus \bar{B}_a(0))$ with $c>0$, $n>1$ and $k\geq 0$
or $Z=\mS_c^{m-1}\times [a,\infty)$ for $c>0$.
In all those cases $(M_Z, g)$ is conformally compactifiable to a manifold $(\widetilde{M}, \tilde{g})$ by considering $\tilde{g}=h^2 g$ where $h(r)=\cosh^{-1}(r)$ for $r\geq a$, cp. \cite[Prop.~3.1]{ADH}, and where $\widetilde{M}\setminus M_Z$ is a totally geodesic $\mS^k$. Note that on a neighbourhood $U$ of the $\mS^k$ the metric $\tilde{g}$ is the standard metric on $\mS^m$. Set $s={\rm dist}_{\tilde{g}}(., \mS^k)$. Then, $\cosh^{-1}(r)=\sin s$ for $r\geq  a$.

Let $\tilde{u}=h^{-\frac{m-2}{2}}u$. Then, $L_{\tilde{g}}\tilde{u}=\delta_{p_0}$ on $\widetilde{M}\setminus \mS^k$. Since $u\in L^2(Z, g)$, $\int_{Z} \frac{1}{\sin^2 s}\tilde{u}^2\,\vo_{\tilde{g}}= \int_{Z} {u}^2\,\vo_{g}<\infty$ and, hence, $\int_{Z} \frac{1}{s^2}\tilde{u}^2\,\vo_{\tilde{g}}<\infty$. Thus, analogously as in \cite[Lem.~7.3]{ammann.grosse:p13b} we can remove the singularity. Thus, $\tilde{u}$ is the Green function of $(\widetilde{M}, \tilde{g})$ around $p_0$ and the mass of $(\widetilde{M}, \tilde{g})$ is zero. It follows that  $(\widetilde{M}\setminus \{p_0\}, \tilde{u}^{\frac{4}{m-2}}\tilde{g})$ is  flat, complete and asymptotically Euclidean. Thus, it has to be $\mR^m$. Thus, $M$ is conformally equivalent to $\mS^m$, i.e., there is a diffeomorphism $\phi\colon M\setminus\to \mS^m$ with $\phi^*\sigma^m= f^2g$ for some $f\in C^\infty(M,\mR_{>0})$. By Liouville's theorem  $\phi|_U\colon U\subset \mS^m\to \phi(U)\subset \mS^m$ is the restriction of a conformal transformation $\psi\colon \mS^m\to\mS^m$. Then, the diffeomorphism $\psi^{-1}\circ \phi\colon M\to \mS^m$ is the identity on $U$. Thus, on $U$ we have $\sigma^m= f^2\sigma^m$ and, hence, $f|_U=1$. Summarizing we obtain that $M_Z$ is conformally equivalent to $\mS^{m-k-1}\times \mH^{k+1}$.
\end{proof}

\begin{example}\label{ex_assump} Let $(M_Z^m= N^n\times \mH_c^{k+1}, g=g_N+g_c)$ where $(\mH_c^{k+1}=\mS^k\times \mR, g_c= \sinh_c(r)^2 \sigma^k+\d\, r^2)$ where $\sinh_c(r):= \frac{1}{c}\sinh(cr)$ for $c>0$ and $\sinh_0(r):=r$. Moreover, we assume that $\scal_N$ is constant, and that $N$ is closed.\\
Then, $L_g=\Delta_N+\Delta_{\mH_c^{k+1}}+\frac{\scal_N - c^2k(k+1)}{a_m}$. Since the scalar curvature of $M_Z$ is constant, we get that the spectrum of $L_g$ is given by $[d,\infty)$ with $d:=\frac{c^2k^2}{4}+\frac{\scal_N - c^2k(k+1)}{a_m}$. Thus, $L_g$ is invertible if and only if $d>0$ which is equivalent to \eqref{cond_main_1}. In particular, if $n>1$, then $\scal_N>0$, and if $n\in \{0,1\}$ this is not possible. Let now $n>1$. Since the whole real line is the spectrum of the Dirac operator $\mH_c^{k+1}$, the Dirac operator on $M_Z$ is invertible if and only if the Dirac operator on $N$ is. This is automatically fulfilled since $\scal_N>0$. 
Hence, 
if \begin{equation*} \frac{kc(3-m)}{m-2}+\frac{2}{m-2}\sqrt{b+\frac{\scal_N}{a_m}} <2\lambda_N\quad \text{with}\ b=\frac{kc^2(1-n)}{4(m-1)}\end{equation*}
is fulfilled, then Theorem~\ref{mainthm_2} applies to $M_Z= N\times \mH_c^{k+1}$. 
Since $m\geq 3$ and $b\leq 0$, 
we always have
\begin{equation*}  \frac{kc(3-m)}{m-2}+\frac{2}{m-2}\sqrt{b+\frac{\scal_N}{a_m}}\leq \frac{2}{m-2}\sqrt{\frac{\scal_N}{a_m}}\leq \frac{2}{\sqrt{(m-2)(m-1)}}\lambda_N<2\lambda_N
 \end{equation*}
where the second inequality uses the Schr\"odinger-Lichnerowicz formula for $N$.

In particular, for $M=\mH_c^{k+1}\times\mS^n$, $n>1$, and $m\geq 3$, all the assumptions of Theorem~\ref{mainthm_2} are fulfilled.
\end{example}

\begin{proof}[Proof of Theorem~\ref{mainthm}]
 Theorem~\ref{mainthm} follows directly from Theorem~\ref{mainthm_2} and Example~\ref{ex_assump}. 
\end{proof}

\section{Application to the Yamabe invariant}

\begin{lemma}\label{schoen_argu} Let $(M_Z^m,g)$ be as in Subsection~\ref{subsec_Green} and we use the notations therein. Assume that the Dirac operator on $(N, g_N)$ is invertible. Let $m=3,4$ or $5$. Let  $L_g$ be an invertible operatora and assume that  $\eqref{cond_main}$ holds.
  Then, $\Q(M_Z,g)<\Q(\mS^m, \sigma^m)$ unless $M_Z$ is conformally equivalent to  $\mS^n\times \mH^{k+1}$.
\end{lemma}

\begin{proof} By Theorem~\ref{mainthm} the mass is positive. In order to obtain $\Q(M_Z)<\Q(\mS^m)$ we then use the test function $\phi$ that was constructed by Schoen out of $\Gamma$, cp. \cite[p. 482]{schoen_84}. Since $\Gamma\in H_1^2$, the calculation remains completely the same.
\end{proof}

Note that the conformal Laplacian of $\mS^n\times \mH^{k+1}$ is invertible only if $n>1$.

\begin{repcorollary}{repcor} Let $m=n+k+1$, $m\geq 3$ $k>0$, and $c\in [0,1)$. Then
\[\Q(\mS^n\times \mH_c^{k+1},\sigma^n+g_c)<\Q(\mS^m, \sigma^m).\]
\end{repcorollary}

\begin{proof} For $m\geq 6$, $c\in [0,1)$ and $n>1$ the manifold $\mS^n\times \mH_c^{k+1}$ is not conformally flat. Thus, Aubin's construction mentioned in the introduction yields the result in this case. For $n>1$ and $m=3,4,5$ this follows from Theorem~\ref{schoen_argu} together with Example~\ref{ex_assump}.
For the remaining case $n=1$ the claim follows from $\Q(\mS^1\times \mH_c^{m-1},\sigma^n+g_c)=c^{\frac{2}{m}}\Q(\mS^m, \sigma^m)$, cp. \cite[Rem.~9.9]{ammann.grosse:p13b}.\end{proof}

\begin{appendix}

\section{Conformally flat Riemannian products}

\begin{lemma}\label{confflat_prod}
 Let $(M^m=M_1^{m_1}\times M_2^{m_2}, g_1+g_2)$ be a product manifold. Let $m_1,m_2\geq 1$ and $m=m_1+m_2\geq 3$. 
  
 Then, $M$ is conformally flat if and only if $M_1$ and $M_2$ both have constant sectional curvature and in case that $m_1, m_2>1$ these sectional curvature have the same magnitude and opposite sign.
 \end{lemma}

\begin{proof} For the only if direction we start with $m\geq 4$. Recall, that an $m\geq 4$-dimensional manifold is conformally flat if and only if its Weyl tensor 
 \[ W=R_g+\frac{\scal_g}{2(m-1)(m-2)} g\kuno g -\frac{1}{m-2} \Ric_g\kuno g\]
 vanishes. We denote the scalar, Ricci, and Riemannian curvature of $M_i$ by $\scal_i$, $\Ric^i$, and $R_i$, respectively.  Let $X\in TM_1$ and $Y\in TM_2$. We compute $W(X,Y,X,Y)$ and obtain
 
 \begin{equation}\label{prod_einst} 0= \frac{\scal_1+\scal_2}{(m-1)(m-2)}-\frac{1}{m-2}\left(\frac{ \Ric^1(X,X)}{\Vert X\Vert^2} +\frac{ \Ric^2(Y,Y)}{\Vert Y\Vert^2} \right).
 \end{equation}
 
 Summation over a basis of $TM_1$ and $TM_2$ gives
 
\[0= m_2(m_2-1)\scal_1+m_1(m_1-1)\scal_2.
 \]
 
 Thus, if $m_1, m_2>1$, then $\scal_i$ is constant for $i=1,2$ and $\scal_2=\frac{m_2(m_2-1)}{m_1(m_1-1)}\scal_1$.
 Moreover, then by \eqref{prod_einst} the $M_i$'s have constant Ricci curvature.  Taking now $X,Y\in TM_i$, $i=1$ or $2$, such that $X\perp Y$, we obtain by considering $W(X,Y,X,Y)$ that 
 $R_i(X,Y,X,Y)=\rm{const}\, g_i\kuno g_i(X,Y,X,Y)$. Thus, both $M_i$ even have constant sectional curvature.

In case $m_1=1$, then the curvatures of $M_1$ vanish and we obtain from \eqref{prod_einst} that $M_2$ is Einstein. If, moreover, we take $X,Y\in TM_2$ and obtain for $W(X,Y,X,Y)$ that
$R_2(X,Y,X,Y)=\frac{\scal_2}{2m(m-1)} g_2\kuno g_2(X,Y,X,Y)$ and, thus, that the sectional curvature is constant.
Thus we conclude the only if direction of the claim for $m\geq 4$.

It remains the case, that $m=3$. Then, $M$ is conformally flat if and only if its Cotton tensor vanishes. The Cotton tensor is given by $C_{ijk}=\nabla_k\left(\Ric^g_{ij}-\frac{1}{2(m-1)} \scal_g g_{ij}\right)-\nabla_j\left(\Ric^g_{ik}-\frac{1}{2(m-1)} \scal_g g_{ik}\right)$. W.l.o.g. let $m_1=1$ and $m_2=2$. Then, $\scal_g=\scal_2$, $\Ric^g=\Ric^2=\frac{1}{2}\scal_2g_2$. Let $\{\partial_i=\partial_j,\partial_k\}$ be a local orthonormal frame of $M_2$. Then
$0=C_{ijk}=\frac{1}{4}\partial_k\scal_2$. Thus, $\scal_2$ is constant.

The if-direction is checked analogously.
\end{proof}

\end{appendix}

\bibliographystyle{acm}

\end{document}